\documentclass[10pt, letterpaper]{article}
\usepackage{graphicx}
\usepackage{amsmath}
\usepackage{amssymb}
\usepackage{color}
\usepackage{multicol}
\usepackage{amsthm}
\usepackage{mathtools} 
\usepackage{tensor} 
\usepackage{hyperref}
\newtheorem{conj}{Conjecture}

\newtheorem{rmk}{Remark}
\newtheorem{theorem}{Theorem}
\newtheorem{proposition}{Proposition}
\newtheorem{lemma}{Lemma}
\begin{document}
\title{The Gauss--Bonnet--Chern mass of higher codimension graphical manifolds
}
\author{Alexandre de Sousa\textsuperscript{1} \and Frederico Gir\~ao\textsuperscript{2}}
\footnotetext[1]{Alexandre de Sousa was a CAPES/Brazil doctoral fellow at Universidade Federal do Cear\'a.}
\footnotetext[2]{Frederico Gir\~ao was partially supported by CNPq/Brazil, grant number 483844/2013-6.}

\date{}
\maketitle
\begin{abstract}
We give an explicit formula for the Gauss--Bonnet--Chern mass of an asymptotically flat graphical manifold of arbitrary codimension and use it to prove the positive mass theorem and the Penrose inequality for graphs with flat normal bundle.
\end{abstract}
\section{Introduction}
\label{intro}
A complete Riemannian manifold $(M^n,g)$, $n \geq 3$, is said to be asymptotically flat of order $\tau$ (with one end) if there exists a compact subset $K$ of $M$ and a diffeomorphism $\Psi: M \setminus K \to \mathbb{R}^n \setminus \overline{B}_1(0)$, introducing coordinates in $M \setminus K$, say $x = (x_1, \ldots, x_n)$, such that, in these coordinates, 
\begin{equation} \label{decay1}
g_{ij} = \delta_{ij} + \sigma_{ij}
\end{equation}
and
\begin{equation} \label{decay2}
|\sigma_{ij}| + |x||\sigma_{ij,k}| + |x|^2|\sigma_{ij,kl}| = O(|x|^{-\tau}),
\end{equation}
where the $\sigma_{ij}$'s are the coefficients of $\sigma$ with respect to $x$, $\sigma_{ij,k} = \partial \sigma_{ij} / \partial x_k$, $\sigma_{ij,kl} = \partial^2 \sigma_{ij} / \partial x_k \partial x_l$, and $| \ |$ is the standard Euclidean norm. The ADM mass of $(M,g)$, introduced by Arnowitt, Deser and Misner in \cite{ADM} (see also \cite{GJ}) is defined by
\begin{equation} \label{ADM_mass}
\mathrm{m}_{\mathrm{ADM}} = \frac{1}{2(n-1)\omega_{n-1}} \lim_{r \to \infty} \int_{S_r} (g_{ij,i} - g_{ii,j})\nu^j dS_r,
\end{equation}
where $\omega_{n-1}$ is the volume of the unit sphere of dimension $(n-1)$, $S_r$ is the Euclidean coordinate sphere of radius $r$, $dS_r$ is the volume form of $S_r$ induced by the Euclidean metric, and $\nu = r^{-1}x$ is the outward unit normal to $S_r$ (with respect to the Euclidean metric).

It is known that if ${\tau > (n-2)/2}$ and the scalar curvature of $(M,g)$ is integrable, then the limit (\ref{ADM_mass}) exists, is finite, and is a geometric invariant, that is,  two coordinate systems satisfying (\ref{decay1}) and (\ref{decay2}) yield the same value for it \cite{B,C}.

One of the most important conjectures in Mathematical General Relativity is the famous Positive Mass Conjecture (PMC):

\begin{conj}
If $(M^n,g)$, $n \geq 3$, is an asymptotically flat Riemannian manifold of order $\tau > (n-2)/2$ whose scalar curvature is nonnegative and integrable, then the ADM mass of $(M,g)$ is nonnegative. Moreover, if the mass is zero, then $(M,g)$ is isometric to the Euclidean space $(\mathbb{R}^n,\delta)$.
\end{conj}

The PMC was settled by Schoen and Yau when $n \leq 7$ \cite{SY1} and when $(M,g)$ is conformally flat \cite{SY2}, and by Witten when $M$ is spin \cite{W} (see also \cite{PT} and \cite{CB}). Very elegant proofs for the case when $(M,g)$ is an Euclidean graph were given by Lam \cite{L} (see also \cite{deLG}) for graphs of codimension one and by Mirandola and Vit\'orio \cite{MV} for graphs of arbitrary codimension with flat normal bundle (notice that the case of graphs also follows from Witten's argument, since an Euclidean graph is spin). The case of Euclidean hypersurfaces (not necessarily graphs), including the rigidity statement, was treated in \cite{HW1}, under appropriate decay conditions.

The Penrose Inequality (PI) is a conjectured sharpening of the PMC when $(M,g)$ has a compact boundary $\Gamma$ which is an outermost minimal hypersurface.
\begin{conj}
If $(M^n,g)$, $n \geq 3$, is an asymptotically flat Riemannian manifold of order $\tau > (n-2)/2$ whose scalar curvature is nonnegative (and integrable), and $\Gamma$ is a (possibly disconnected) outermost minimal hypersurface 
of area $A$, then
$$\mathrm{m}_{\mathrm{ADM}} \geq \frac{1}{2}\left(\frac{A}{\omega_{n-1}} \right)^{\frac{n-2}{n-1}}.$$
Moreover, if the equality holds, then $(M,g)$ is isometric to the Riemannian Schwarzschild manifold.
\end{conj}

The PI was proved by Huisken and Ilmanen \cite{HI} for $n=3$ and $\Gamma$ connected, and by Bray \cite{Bray} for $n=3$ and general $\Gamma$. In \cite{Bray-Lee}, Bray and Lee established the conjecture when $n \leq 7$, with the extra requirement that $M$ be spin for the rigidity statement. The case of Euclidean graphs of codimension one was treated by Lam in \cite{L} (see also \cite{deLG}) and generalized by Mirandola and Vit\'orio for graphs of arbitrary  codimension with flat normal bundle \cite{MV}. The equality case for graphs of codimension one was treated in \cite{HW2}.

In \cite{GWW}, a new mass for asymptotically flat Riemannian manifolds, named Gauss--Bonnet--Chern mass, was introduced. For a positive integer $q \leq n/2$, consider the $q$-th Gauss--Bonnet curvature, denoted $L_{(q)}$, and defined by
\begin{equation}
  \label{eq:gbc curvature}
  L_{(q)} = \frac{1}{2^q}\delta_{b_1b_2 \cdots b_{2q}}^{a_1a_2 \cdots a_{2q}}\left( \prod_{s=1}^{q} R\indices{_{a_{2s-1}a_{2s}}^{b_{2s-1}b_{2s}}} \right) = P_{(q)}^{ijkl}R_{ijkl},
\end{equation}
where $R$ is the Riemann curvature tensor of $(M,g)$ and $P_{(q)}$, which has the same symmetries of the Riemann tensor (see \cite{GWW}, Section 3), is given by
\begin{equation} \label{P}
  P_{(q)}^{ijkl} = \frac{1}{2^q}\delta_{b_1b_2 \cdots b_{2q-3}b_{2q-2}b_{2q-1}b_{2q}}^{a_1a_2 \cdots a_{2q-3}a_{2q-2}ij}
  \left( \prod_{s=1}^{q-1} R\indices{_{a_{2s-1}a_{2s}}^{b_{2s-1}b_{2s}}} \right)g^{b_{2q-1}k}g^{b_{2q}l}.
\end{equation}
\begin{rmk}
{ \rm One can considerably simplify this complicated tensorial expression by rewriting it in the language of double forms, 
 which are a special type of vector valued forms 
(see \cite{La}, for example).}
\end{rmk}

The $q$-th Gauss--Bonnet--Chern mass (GBC mass) of $(M,g)$ is defined by
\begin{equation} \label{GBC mass}
\mathrm{m}_q = c_q(n) \lim_{r \to \infty} \int_{S_r} P_{(q)}^{ijkl}g_{jk,l}\nu_i dS_r,
\end{equation}
where 
\begin{equation} \label{eq:constant GBC mass}
c_q(n) = \frac{(n-2q)!}{2^{q-1}(n-1)!\omega_{n-1}}
\end{equation}
and $S_r$, $dS_r$, $\nu$ and $\omega_{n-1}$ are as in the definition of the ADM mass.

As observed in \cite{GWW}, $\mathrm{m}_1$ coincides with the ADM mass. In the same article, the authors show that, if $\tau > (n-2q)/(q+1)$ and $L_{(q)}$ is integrable, then the limit (\ref{GBC mass}) exists, is finite, and is a geometric invariant. Next, we state versions of the PMC and PI for the GBC mass. We start with the version of the PMC.

\begin{conj} \label{PMT for m_q}
Let $n$ and $q$ be integers such that $n \geq 3$ and $ 1 \leq q < n/2$. If $(M^n,g)$ is an asymptotically flat Riemannian manifold of order $\tau > (n-2q)/(q+1)$ whose $q$-th Gauss--Bonnet curvature $L_{(q)}$ is nonnegative and integrable, then the $q$-th GBC mass of $(M,g)$ is nonnegative. Moreover, if the mass is zero, then $(M,g)$ is isometric to the Euclidean space $(\mathbb{R}^n,\delta)$.
\end{conj}

Before we state the analogue of the PI, we recall the Riemannian manifold known as the $q$-th Riemannian Schwarzschild \cite[Section~6]{GWW}, which is $(\mathbb{R}\times \mathbb{S}^{n-1}, g^q_{\mathrm{Sch}})$ with
$$
g^q_{\mathrm{Sch}} = \left( 1 + \frac{m}{2r^{\frac{n}{q} - 2}} \right)^{\frac{4q}{n-2q}} \left( dr^2 + r^2 d\theta^2 \right),
$$
where $d\theta^2$ is the round metric on $\mathbb{S}^{n-1}$ and $m \in \mathbb{R}$ is the mass parameter. Let $r_0 = (2m)^{\frac{q}{n-2q}}$. The hypersurface $r = r_0$ is an outermost minimal hypersurface of area $A=\omega_{n-1}r_0^{n-1}$, and the $q$-th GBC mass of $(\mathbb{R}\times \mathbb{S}^{n-1}, g^q_{\mathrm{Sch}})$ is $\mathrm{m}_q = m^q$. Thus, for the $q$-th Riemannian Schwarzschild manifold, one has
$$ \mathrm{m}_q = \frac{1}{2^q} \left( \frac{A}{\omega_{n-1}} \right)^{\frac{n-2q}{n-1}}.$$
We can now state the version of the PI for the GBC mass.

\begin{conj} \label{PI for m_q}
Let $n$ and $q$ be integers such that $n \geq 3$ and $ 1 \leq q < n/2$. If $(M^n,g)$ is an asymptotically flat Riemannian manifold of order $\tau > (n-2q)/(q+1)$ whose $q$-th Gauss--Bonnet curvature $L_{(q)}$ is nonnegative and integrable, and $\Gamma$ is a (possibly disconnected) outermost minimal hypersurface 
of area $A$, then
$$\mathrm{m}_{q} \geq \frac{1}{2^q}\left(\frac{A}{\omega_{n-1}} \right)^{\frac{n-2q}{n-1}},$$
where $\mathrm{m}_q$ is the $q$-th GBC mass.
Moreover, if the equality holds, then $(M,g)$ is isometric to the $q$-th Riemannian Schwarzschild manifold.
\end{conj}

We now turn to the special case of graphs. Let $\Omega$ be a (possibly empty) bounded open subset of $\mathbb{R}^n$ such that $\Sigma = \partial \Omega$ is the union of finitely many smooth connected hypersurfaces. Let $f: \mathbb{R}^n \setminus \Omega \to \mathbb{R}^m$ be a continuous map such that its restriction to $\mathbb{R}^n \setminus \overline{\Omega}$ is smooth. Let $f^{\alpha}$, $1 \leq \alpha \leq m$, be the components of $f$ and let $f^{\alpha}_i$, $f^{\alpha}_{ij}$ and $f^{\alpha}_{ijk}$ denote the first, second and third partial derivatives of $f^{\alpha}$ on $\mathbb{R}^n \setminus \overline{\Omega}$, where $1 \leq i,j,k \leq n$. The map $f$ is said to be asymptotically flat of order $\tau$ if 
\begin{equation} \label{AF function}
|f^{\alpha}_i(x)| + |f^{\alpha}_{ij}(x)| |x| + |f^{\alpha}_{ijk}(x)| |x|^2 = O(|x|^{-\tau/2}),
\end{equation}
for each $\alpha \in \lbrace 1, \ldots, m \rbrace$ and each triple $(i,j,k)$ with $ 1 \leq i, j, k \leq n$.

We assume throughout the paper that $$M = \lbrace (x,f(x)); \ x \in \mathbb{R}^n \setminus \Omega, f(x) \in \mathbb{R}^m \rbrace,$$ the graph of $f$, is a smooth submanifold with (possibly empty) boundary and that $g_{ij} = \delta_{ij} + f^{\alpha}_if^{\alpha}_j$, the metric induced by the Euclidean metric on $\mathbb{R}^{n+m}$, extends to a smooth metric on $M$. Notice that if $f$ is asymptotically flat of order $\tau$, then from (\ref{AF function}) we get that $(M,g)$ is asymptotically flat of order $\tau$.

Conjectures \ref{PMT for m_q} and \ref{PI for m_q} were proved for graphs of codimension one \cite{GWW}. When $q=2$, Li, Wei and Xiong proved these conjectures for graphs of higher codimension with flat normal bundle \cite{LWX}. Conjecture \ref{PMT for m_q} is also known to be true for conformally flat manifolds \cite{GWW2}.

The purpose of the present article is to prove conjectures \ref{PMT for m_q} and \ref{PI for m_q} for a family of higher codimension Euclidean graphs (without the rigidity statements). This family includes the graphs with flat normal bundle. The exposition follows closely the ones given in \cite{L}, \cite{MV}, \cite{GWW} and \cite{LWX}. Before stating our main results, we need to introduce some notation.

Denote by $\lbrace e_i \rbrace_{i=1}^n$ the standard basis of $\mathbb{R}^n$ and by $\lbrace e_{\alpha} \rbrace_{\alpha = 1}^m$ the standard basis of $\mathbb{R}^m$. The coordinate vector fields on $M$ are given by $\partial_i = (e_i, f^{\alpha}_i e_{\alpha})$, and the vector fields $\eta_{\alpha} = (-Df^{\alpha},e_{\alpha})$, where $Df^{\alpha}$ denotes the Euclidean gradient of $f^{\alpha}$, give us a (global) frame field for the normal bundle of $M$. We denote by $B$ the second fundamental form of $M$, by $B_{\alpha}$ its $\alpha$-th component with respect to the frame $\lbrace \eta_{\alpha} \rbrace_{\alpha=1}^{n}$, and by $A_{\alpha}$ the shape operator with respect to $\eta_{\alpha}$. Also, let $U = (U_{\alpha \beta})$ be the metric on the normal bundle induced by the Euclidean metric $\langle \ , \hspace{1pt} \rangle$ on $\mathbb{R}^{n+m}$. The components of $U$ are given by
$$ U_{\alpha \beta} = \delta_{\alpha \beta} + \langle Df^{\alpha}, Df^{\beta} \rangle. $$
The inverse of $U$ is denoted by $(U^{\alpha \beta})$.

Recall the Gauss and the Ricci equations, which are respectively given by
\begin{equation} \label{Gauss}
R_{ijkl} = \langle B_{ik},B_{jl} \rangle - \langle B_{il},B_{jk} \rangle
\end{equation}
and 
\begin{equation} \label{Ricci}
\langle R^{\perp}_{\alpha \beta} (X), Y \rangle = \langle \left[A_{\beta},A_{\alpha}\right](X),Y \rangle,
\end{equation}
where $\langle \ , \hspace{1pt} \rangle$ is the Euclidean metric on $\mathbb{R}^{n+m}$, $R^{\perp}$ is the normal curvature operator and $[A_{\beta},A_{\alpha}] = A_{\alpha} \circ A_{\beta} - A_{\beta} \circ A_{\alpha}.$

We denote by $T_{(2q-1)}$ the Newton tensor of order $(2q-1)$ and denote by $T_{(2q-1)\alpha}$ its $\alpha$-th component with respect to the frame $\lbrace \eta_{\alpha} \rbrace_{\alpha=1}^{n}$ (see 
\cite{G} and  \cite{CL}). The expression for $T_{(2q-1)}$ in coordinates is
\begin{equation} \label{T}
  T\indices{_{(2q-1)}_i^j}= \frac{1}{(2q-1)!}\delta^{a_1\cdots a_{2q-1}j}_{b_1 \cdots b_{2q-1}i}
\langle B_{a_1}^{b_1}, B_{a_2}^{b_2} \rangle \cdots \langle B_{a_{2q-3}}^{b_{2q-3}}, B_{a_{2q-2}}^{b_{2q-2}} \rangle B_{a_{2q-1}}^{b_{2q-1}},
\end{equation}
where $\langle \ , \hspace{1pt} \rangle$ denotes the Euclidean metric on $\mathbb{R}^{n+m}$.
As we will see in Section \ref{sec:1}, if $M$ has flat normal bundle, then $T_{(2q-1)\alpha}$ commutes with $A_{\beta}$, for $1 \leq \alpha, \beta \leq m$.

We can now state the main results of the article. Our first main result is the following:
\begin{theorem} \label{main 1}
Let $n$ and $q$ be integers such that $n \geq 3$ and $ 1 \leq q < n/2$, and let $(M,g)$ be the graph of an asymptotically flat map $f: \mathbb{R}^n \to \mathbb{R}^m$ of order 
$\tau > (n-2q)/(q+1)$. 
If the $q$-th Gauss--Bonnet curvature $L_{(q)}$ of $(M,g)$ is integrable, then the $q$-th Gauss--Bonnet--Chern mass $\mathrm{m}_q$ satisfies
\begin{equation} \label{eq main 1}
  \mathrm{m}_q = \frac{1}{2} c_q(n) \int_M \left(L_{(q)} + (2q-1)! \left\langle \left[ T_{(2q-1)\alpha}, A_{\beta} \right] \cdot e_{\alpha}^{\top}, e_{\beta}^\top \right\rangle \right) \frac{1}{\sqrt{G}}  dM,
\end{equation}
where $c_q(n)$ is the constant \eqref{eq:constant GBC mass}, $G$ is the determinant of $(g_{ij})$, $\left[T_{(2q-1)\alpha},A_{\beta}\right] = T_{(2q-1)\alpha} \circ A_{\beta} - A_{\beta} \circ T_{(2q-1)\alpha}$ is the commutator of the operators $T_{(2q-1)\alpha}$ and $A_{\beta}$, and $e_{\alpha}^{\top}$ is the tangent part (along the graph $M$) of the canonical lift to $\mathbb{R}^{n+m} \equiv \mathbb{R}^n \times \mathbb{R}^m$ of the standard frame field on $\mathbb{R}^m$. 
Moreover, if $M$ has flat normal bundle and $L_{(q)}$ is nonnegative, then $\mathrm{m}_q$ is nonnegative.
\end{theorem} 

\begin{rmk}
{\rm Notice that, since the graph structure is used in the definition of the vector fields $\eta_\alpha$, the tensor
$ \left[ T_{(2q-1)\alpha}, A_{\beta} \right] $
is defined only for graphs. It is desirable to find an expression similar to (\ref{eq main 1}) that holds for any asymptotically flat submanifold (not necessarily a graph), but we were unable to do it. One strategy in order to do this is to rewrite $ \left[ T_{(2q-1)\alpha}, A_{\beta} \right] $ in such a way that it also makes sense for submanifolds which are not necessarily graphs and try proving that (\ref{eq main 1}) also holds in this case. Another strategy is to find a similar expression for the mass by considering, instead of the vector field (\ref{eq:field}) used to get (\ref{eq main 1}), one which is defined for any submanifold (compare, for example, the vector fields considered in \cite{L} and \cite{deLG}).}
\end{rmk}

Let $\Sigma \subset \mathbb{R}^n$ be an orientable hypersurface and let $\xi$ be a unit normal vector field along $\Sigma$ (chosen to point outwards, whenever this makes sense). The $r$-th mean curvature of $\Sigma$ is defined as the $r$-th elementary symmetric function on the principal curvatures of $\Sigma$. Alternatively, if $K$ is the second fundamental form of $\Sigma \subset \mathbb{R}^n$, then
\begin{equation}
  \label{eq:r-mean-curvature}
  H_r = \frac{1}{r!} \delta^{a_1 \cdots a_r}_{b_1 \cdots b_r} \prod_{s=1}^r K_{a_s}^{b_s}.
\end{equation}

\label{par:estritamenteComvexo} The hypersurface $\Sigma \subset \mathbb{R}^n$ is called \textit{strictly $p$-mean convex}, $1 \leq p \leq n-1$, if $H_r > 0$ for all $1 \leq r \leq p$. Our second main result is the following:


\begin{theorem} \label{main 1.5}
Let $n$ and $q$ be integers such that $n \geq 3$ and $ 1 \leq q < n/2$.
Let $\Omega$ be a bounded and open subset of $\mathbb{R}^n$ such that $\Sigma = \partial \Omega$ is the union of finitely many smooth connected hypersurfaces. Let $f: \mathbb{R}^n \setminus \Omega \to \mathbb{R}^m$ be an  asymptotically flat map of order 
$\tau > (n-2q)/(q+1)$, 
and let $(M,g)$ be the graph of $f$. Assume that $f$ extends smoothly to an open set containing $\mathbb{R}^n \setminus \Omega$ and that $f$ is constant along each connected component of $\Sigma$. 
If the $q$-th Gauss--Bonnet curvature $L_{(q)}$ is integrable, then the $q$-th Gauss--Bonnet--Chern mass $\mathrm{m}_q$ satisfies
\begin{align}
  \mathrm{m}_q = &  \frac{1}{2} c_q(n) \int_M \left(L_{(q)} + (2q-1)! \left\langle \left[ T_{(2q-1)\alpha}, A_{\beta} \right] \cdot e_{\alpha}^{\top}, e_{\beta}^{\top} \right\rangle \right) \frac{1}{\sqrt{G}}  dM \nonumber \\
& + \frac{1}{2} (2q-1)! c_q(n) \int_{\Sigma} \left( \frac{|Df|^2}{1 + |Df|^2} \right)^q H_{(2q-1)} d\Sigma, \nonumber
\end{align}
where 
$$ |Df|^2 = \sum_{\alpha=1}^m |Df^{\alpha}|^2$$
and $H_{(2q-1)}$ is the $(2q-1)$-th mean curvature of $\Sigma$.
\end{theorem}

Our third main result is the following:

\begin{theorem} \label{main 2}
Let $n$ and $q$ be integers such that $n \geq 3$ and $ 1 \leq q < n/2$.
Let $\Omega$ be a bounded and open subset of $\mathbb{R}^n$ such that $\Sigma = \partial \Omega$ is the union of finitely many smooth hypersurfaces. Let $f: \mathbb{R}^n \setminus \Omega \to \mathbb{R}^m$ be an  asymptotically flat map of order 
$\tau > (n-2q)/(q+1)$, 
and let $(M,g)$ be the graph of $f$. Assume that $f$ is constant along each connected component of $\Sigma$ and that 
$$ | Df | \to \infty \ \ \text{as} \ \ x \to \Sigma.$$
If the $q$-th Gauss-Bonnet curvature $L_{(q)}$ is integrable, then the $q$-th Gauss-Bonnet-Chern mass $\mathrm{m}_q$ satisfies
\begin{align}
  \mathrm{m}_q = &  \frac{1}{2} c_q(n) \int_M \left(L_{(q)} + (2q-1)! \left\langle \left[ T_{(2q-1)\alpha}, A_{\beta} \right] \cdot e_{\alpha}^{\top}, e_{\beta}^{\top} \right\rangle \right) \frac{1}{\sqrt{G}}  dM \nonumber \\
& + \frac{1}{2} (2q-1)! c_q(n) \int_{\Sigma} H_{(2q-1)} d\Sigma, \label{eq main 2}
\end{align}
where $H_{(2q-1)}$ is the $(2q-1)$-th mean curvature of $\Sigma$.
Furthermore, if $M$ has flat normal bundle, $L_{(q)}$ is nonnegative and each component of $\Sigma$ is star-shaped and strictly $(2q-1)$-mean convex, then
\begin{equation} \label{eq_2 main 2}
\mathrm{m}_{q} \geq \frac{1}{2^q}\left(\frac{A}{\omega_{n-1}} \right)^{\frac{n-2q}{n-1}}.
\end{equation}
\end{theorem}

\begin{rmk}
{\rm As explained in \cite{GWW} (Remark 5.1), when $\Sigma \subset \mathbb{R}^n$ is stricly mean convex, the condition
\begin{equation}\label{condition}
 |D f| \to \infty \ \ \text{as} \ \ x \to \partial \Omega
 \end{equation}
holds if and only if $\Gamma = \partial M$ is an outermost minimal hypersurface. Therefore, this is a natural assumption.}
\end{rmk}

\begin{rmk}
{\rm Geometrically, condition (\ref{condition}) is equivalent to saying that along each connected component of $\partial M$, the graph $M$ meets orthogonally the hyperplane that contains that component (see \cite{deS}).}
\end{rmk}

\section{Auxiliary results}
\label{sec:1}
Let $n$ and $q$ be positive integers such that $n \geq 3$ and $1 \leq q < n/2$. Throughout this section, the tensors $P_{(q)}$ and $T_{(2q-1)}$ are defined by equations (\ref{P}) and (\ref{T}), respectively. Also, unless stated otherwise, we will follow the notation introduced in Section \ref{intro}.

If $\Omega$ is not empty, we assume, throughout this section, that $f$ extends smoothly to an open set containing $\mathbb{R}^n \setminus \Omega$.

\begin{lemma} Under the notation introduced above, the following identities hold:
\begin{align}
g_{jk,l} = & \ f^{\alpha}_{jl} f^{\alpha}_k + f^{\alpha}_jf^{\alpha}_{kl} \label{a} \\
e_{\alpha}^{\top} = & \ \nabla f^{\alpha} = g^{ij}f_j^{\alpha} \partial_i = U^{\alpha\beta} f_i^\beta \partial_i \label{c} \\
A_{\alpha} \partial_i = & \ f^{\alpha}_{ik}g^{kj}\partial_i \label{shape_operator}\\
B(\partial_i, \partial_j) = & \ f^{\alpha}_{ij}U^{\alpha \beta} \eta^{\beta} \\
(B_{\alpha})_{ij} = & \ f^{\alpha}_{ij} \label{f} \\
\Gamma_{ij}^k = & \ g^{kl}f^{\alpha}_lf^{\alpha}_{ij} \label{g} \\
\nabla e_\alpha^{\top} = & \ \langle B, e_\alpha \rangle \label{h}
\end{align}
\end{lemma}
\begin{proof}
Identities (\ref{a}) - (\ref{g}) are proven in \cite{MV} and \cite{LWX} and identity (\ref{h}) is proved in \cite{palais_critical_1988} (Proposition 4.1.1).
\end{proof}

On an open set that contains $\mathbb{R}^n \setminus \Omega$, consider the vector field $X_{(q)}$ given by
\begin{equation} \label{vector field}
  X_{(q)} = X_{(q)}^i \partial_i = P_{(q)}^{ijkl}g_{jk,l} \partial_i.
\end{equation}

\begin{proposition}
It holds
\begin{equation}
  \label{eq:field}
  X_{(q)} = \frac{1}{2}(2q-1)! \ T_{(2q-1)\alpha} \cdot e_{\alpha}^{\top}.
\end{equation}

\end{proposition}
\begin{proof}
By (\ref{vector field}) and (\ref{a}) we have
$$ X_{(q)}^i = P_{(q)}^{ijkl} \left( f^{\alpha}_{jl}f^{\alpha}_k + f^{\alpha}_jf^{\alpha}_{kl} \right).$$
Using the antisymmetry of $P_{(q)}^{ijkl}$ with respect to the indices $k$ and $l$, we have
$$X_{(q)}^i =  P_{(q)}^{ijkl} f^{\alpha}_{jl}f^{\alpha}_k .$$
Combining this identity with \eqref{P}, \eqref{shape_operator} and \eqref{c}, we find
\begin{align*}
X_{(q)}^i = & \ \frac{1}{2^q}\delta_{b_1b_2 \cdots b_{2q-3}b_{2q-2}cd}^{a_1a_2 \cdots a_{2q-3}a_{2q-2}ij}\left( \prod_{s=1}^{q-1} R\indices{_{a_{2s-1}a_{2s}}^{b_{2s-1}b_{2s}}} \right)g^{ck}g^{dl}f^{\alpha}_{jl}f^{\alpha}_k \\
 =  & \ \frac{1}{2^q}\delta_{b_1b_2 \cdots b_{2q-3}b_{2q-2}cd}^{a_1a_2 \cdots a_{2q-3}a_{2q-2}ij}\left( \prod_{s=1}^{q-1} R\indices{_{a_{2s-1}a_{2s}}^{b_{2s-1}b_{2s}}} \right)g^{dl}f^{\alpha}_{jl}g^{ck}f^{\alpha}_k \\
= & \ \frac{1}{2^q}\delta_{b_1b_2 \cdots b_{2q-3}b_{2q-2}cd}^{a_1a_2 \cdots a_{2q-3}a_{2q-2}ij}\left( \prod_{s=1}^{q-1} R\indices{_{a_{2s-1}a_{2s}}^{b_{2s-1}b_{2s}}} \right)\left(A_{\alpha}\right)_j^dg^{ck}f^{\alpha}_k \\
= & \ \frac{1}{2^q}\delta_{b_1b_2 \cdots b_{2q-3}b_{2q-2}cd}^{a_1a_2 \cdots a_{2q-3}a_{2q-2}ij}\left( \prod_{s=1}^{q-1} R\indices{_{a_{2s-1}a_{2s}}^{b_{2s-1}b_{2s}}} \right)\left(A_{\alpha}\right)_j^d\left( \nabla f^{\alpha} \right)^c. 
\end{align*}
Hence, using \eqref{Gauss}, \eqref{T}, \eqref{c} and switching $i$ with $j$ and $c$ with $d$, we find
\begin{align*}
X_{(q)}^i = & \ \frac{2^{q-1}}{2^q}\delta_{b_1b_2 \cdots b_{2q-3}b_{2q-2}cd}^{a_1a_2 \cdots a_{2q-3}a_{2q-2}ij}\left( \prod_{s=1}^{q-1} \langle B^{b_{2s-1}}_{a_{2s-1}} , B^{b_{2s}}_{a_{2s}} \rangle \right)\left(A_{\alpha}\right)_j^d\left( \nabla f^{\alpha} \right)^c \\
= & \ \frac{1}{2}\delta_{b_1b_2 \cdots b_{2q-3}b_{2q-2}dc}^{a_1a_2 \cdots a_{2q-3}a_{2q-2}ji}\left( \prod_{s=1}^{q-1} \langle B^{b_{2s-1}}_{a_{2s-1}} , B^{b_{2s}}_{a_{2s}} \rangle \right)\left(A_{\alpha}\right)_j^d\left( \nabla f^{\alpha} \right)^c \\
= & \ \frac{1}{2}(2q-1)!\left( T_{(2q-1)\alpha} \right)^i_c \left( \nabla f^{\alpha} \right)^c \\
= & \ \frac{1}{2}(2q-1)!\left( T_{(2q-1)\alpha} \cdot \nabla f^{\alpha} \right)^i \\
  = & \ \frac{1}{2}(2q-1)!\left( T_{(2q-1)\alpha} \cdot e_{\alpha}^{\top} \right)^i.
\end{align*}
\end{proof}
The next identity is a higher codimensional version of Proposition 3.5 (b) in \cite{R2} (see also \cite{AdeLM}, Section 8).
\begin{proposition} \label{field_div}
It holds
$$
\mathrm{div}_e X =
\frac{1}{2}L_{(q)} + \frac{1}{2}(2q-1)!  \left\langle \left[ T_{(2q-1)\alpha}, A_{\beta} \right] \cdot e_{\alpha}^{\top}, e_{\beta}^{\top} \right\rangle,
$$
where \( \mathrm{div}_e \) denotes the Euclidean divergence.
\end{proposition}
\begin{proof}
  Using the identity
  \begin{equation*}
    \mathrm{div}_e X_{(q)} = \partial_i X\indices*{_{(q)}^i}
  \end{equation*}
  and the identities \eqref{c}, \eqref{f} and \eqref{g}, we have
  \begin{align*}
    \mathrm{div}_g X_{(q)} &= \nabla_i X_{(q)}^i = \partial_i X\indices*{_{(q)}^i} + \Gamma_{ij}^i X\indices*{_{(q)}^j} \\
                         &= \mathrm{div}_e X_{(q)} + (e_{\beta}^{\top})^i (B_{\beta})_{ij} X\indices*{_{(q)}^j} \\
                         &= \mathrm{div}_e X_{(q)} + \left\langle A_{\beta} \cdot X_{(q)}, e_{\beta}^{\top} \right\rangle \\
                         &= \mathrm{div}_e X_{(q)} + {\frac{1}{2}}(2q-1)! \left\langle \left( A_{\beta} \circ T_{(2q-1)\alpha} \right) \cdot e_{\alpha}^{\top}, e_{\beta}^{\top} \right\rangle.
  \end{align*}
  By the expression for the vector field $X_{(q)}$ established in the previous proposition, it follows that
  \begin{align*}
    \mathrm{div}_g X_{(q)} &= \frac{1}{2} (2q-1)! \ \mathrm{div}_g \left( T_{(2q-1)\beta} \cdot e_{\beta}^{\top} \right) \\
                         &= \frac{1}{2} (2q-1)! \big[ \mathrm{div}_g \left( T_{(2q-1)\beta}  \right) \cdot e_{\beta}^{\top}
                           +  T_{(2q-1)\beta} \cdot \nabla e_{\beta}^{\top} \big]. \label{eq:campo-divergencia-variedade-passo1}
  \end{align*}
By \eqref{c} and (\ref{h}), the identities
  \begin{equation*}
    \nabla e_{\beta}^{\top} 
    = \langle B, e_{\beta}  \rangle = U^{\gamma\alpha} \langle \eta_{\gamma}, e_{\beta} \rangle B_{\alpha} = U^{\beta\alpha} B_{\alpha}
  \end{equation*}
  hold. 
  Therefore, the Gauss equation together with identities \eqref{T} and \eqref{eq:gbc curvature} give
  $$ 
    T_{(2q-1)\beta} \cdot \nabla e_{\beta}^{\top} = U^{\beta\alpha} T_{(2q-1)\beta} \cdot B_\alpha = \frac{1}{(2q-1)!} L_{(q)}.
  $$
  Thus,
  \begin{equation*}
    \mathrm{div}_e X_{(q)} = {\frac{1}{2}} L_{(q)} + \frac{1}{2} (2q-1)! \left[ \mathrm{div}_g \left( T_{(2q-1)\beta}  \right) \cdot e_{\beta}^{\top} - \left\langle \left( A_{\beta} \circ T_{(2q-1)\alpha} \right) \cdot e_{\alpha}^{\top}, e_{\beta}^{\top} \right\rangle \right].
  \end{equation*}
  Recall that the Newton tensors of a submanifold of Euclidean space are divergence free (see, for example, lemmata 3.1 and 3.2 of \cite{CL}) and that each of the fields in the normal frame is given by the expression $\eta_{\beta} = (-Df^{\beta},e_{\beta})$. Hence, using identity \eqref{c}, we have
  \begin{align*}
    \left( \mathrm{div}_g T_{(2q-1)\beta} \right)_j &= \nabla_i \left( T_{(2q-1)\beta} \right)\indices{^i_j}
                                                     = \nabla_i \left\langle \left( T_{(2q-1)} \right)\indices{^i_j}, \eta_{\beta} \right\rangle \\
                                                   &= \left\langle \nabla^\perp_i \left( T_{(2q-1)} \right)\indices{^i_j}, \eta_{\beta} \right\rangle
                                                     + \left\langle \left( T_{(2q-1)} \right)\indices{^i_j}, \nabla^\perp_i \eta_{\beta} \right\rangle \\
                                                   &= \left\langle \left( \mathrm{div} \ T_{(2q-1)} \right)_j, \eta_{\beta} \right\rangle
                                                     + U^{\gamma\alpha} \left( T_{(2q-1)\alpha} \right)\indices{^i_j}
                                                     \left\langle \eta_{\gamma}, {\bar D}_i \eta_{\beta} \right\rangle \\
                                                   &= \left( T_{(2q-1)\alpha} \right)\indices{^i_j}
                                                     U^{\gamma\alpha} f_k^{\gamma} f_{ik}^{\beta}
                                                   = \left( T_{(2q-1)\alpha} \right)\indices{^i_j}
                                                     \left( e_{\alpha}^{\top} \right)^k (B_{\beta})_{ik} \\
                                                   &= \left( \left( T_{(2q-1)\alpha} \circ A_{\beta} \right) \cdot e_{\alpha}^{\top} \right)_j,
  \end{align*}
  where $\bar{D}$ is the Levi-Civita connection of the ambient space $\mathbb{R}^{n+m} \equiv \mathbb{R}^n \times \mathbb{R}^m$. Therefore,
  \begin{align*}
    \mathrm{div}_e X_{(q)} &= {\frac{1}{2}} L_{(q)} + \frac{1}{2} (2q-1)! \left\langle  \left( T_{(2q-1)\alpha} \circ A_{\beta} - A_{\beta} \circ T_{(2q-1)\alpha} \right) \cdot e_{\alpha}^{\top}, e_{\beta}^{\top} \right\rangle \\
                         &= {\frac{1}{2}} L_{(q)} + \frac{1}{2} (2q-1)! \left\langle  \left[ T_{(2q-1)\alpha}, A_{\beta} \right] \cdot e_{\alpha}^{\top}, e_{\beta}^{\top} \right\rangle.
  \end{align*}
\end{proof}

\begin{proposition} \label{prop 3}
  For a level set $\Sigma \subset \mathbb{R}^n$ in the domain of a euclidean graph, the identity
$$ \langle X_{(q)}, \xi \rangle
= - \frac{1}{2}(2q-1)! \left( \frac{|Df|^2}{1 + |Df|^2} \right)^{q}H_{(2q-1)}$$
holds, where $\xi$ denotes a unit normal vector field along $\Sigma$ (chosen to point outwards, whenever this makes sense).
\end{proposition}
\begin{proof}
We have
$$ \langle X_{(q)}, \xi \rangle = \frac{1}{2}(2q-1)! \left( T_{(2q-1)\alpha} \cdot e_{\alpha}^{\top} \right)^i \xi_i.$$
Let $x \in \Sigma$. Rotate the coordinates such that, at $x$, $e_1 = \xi$ and $\lbrace e_A \rbrace_{A=2}^n$ is an orthonormal frame for the tangent space of $\Sigma$ at $x$. With respect to this new frame $\lbrace e_i \rbrace_{i = 1}^n$ on $\mathbb{R}^n$,
$$ \xi_i = \delta_i^1,$$
for $i= 1, \ldots, n$.
Thus,
$$\langle X_{(q)}, \xi \rangle = \frac{1}{2}(2q-1)! \left( T_{(2q-1)\alpha} \cdot e_{\alpha}^{\top} \right)^1.$$
As in section 4 of \cite{LWX}, we find
that the inverse of $g$ is given by
$$ g^{11} = \frac{1}{1+|Df|^2}, \ \ g^{A1} = 0, \textrm{ and } g^{AB} = \delta^{AB}.$$
It follows that
$$ e_{\alpha}^{\top} = \frac{f_1^{\alpha}}{1 + |Df|^2} \partial_1 = \frac{\langle Df^{\alpha}, \xi \rangle}{1+ |Df|^2}\partial_1.$$
Therefore,
$$ \langle X_{(q)}, \xi \rangle = \frac{1}{2}(2q-1)!\frac{\langle Df^{\alpha}, \xi \rangle}{1+ |Df|^2} \left( T_{(2q-1)\alpha} \right)^1_1.$$
Since
$$ \left( T_{(2q-1)\alpha} \right)^1_1 = \frac{1}{2^{q-1}} \frac{1}{(2q-1)!}
\delta^{a_1 \cdots a_{2q-1}1}_{b_1 \cdots b_{2q-1}1} \left( \prod_{s=1}^{q-1} R\indices{_{a_{2s-1}a_{2s}}^{b_{2s-1}b_{2s}}}\right) \left( A_{\alpha} \right)^{b_{2q-1}}_{a_{2q-1}},$$
using the antisymmetry of $\delta^{a_1 \cdots a_{2q-1}1}_{b_1 \cdots b_{2q-1}1}$ we find that
$$\left( T_{(2q-1)\alpha} \right)^1_1 = \frac{1}{2^{q-1}} \frac{1}{(2q-1)!}
\delta^{A_1 \cdots A_{2q-1}1}_{B_1 \cdots B_{2q-1}1} \left( \prod_{s=1}^{q-1} R\indices{_{A_{2s-1}A_{2s}}^{B_{2s-1}B_{2s}}} \right) \left( A_{\alpha} \right)^{B_{2q-1}}_{A_{2q-1}}.$$
Recall that the generalized Kronecker delta is a determinant. Using the $q$-th column to expand it, we find
$$ \delta^{A_1 \cdots A_{2q-1}1}_{B_1 \cdots B_{2q-1}1} = \delta^{A_1 \cdots A_{2q-1}}_{B_1 \cdots B_{2q-1}}.$$
Hence,
$$\left( T_{(2q-1)\alpha} \right)^1_1 = \frac{1}{2^{q-1}} \frac{1}{(2q-1)!}
\delta^{A_1 \cdots A_{2q-1}}_{B_1 \cdots B_{2q-1}} \left( \prod_{s=1}^{q-1} R\indices{_{A_{2s-1}A_{2s}}^{B_{2s-1}B_{2s}}} \right) \left( A_{\alpha} \right)^{B_{2q-1}}_{A_{2q-1}}.$$

Let $\hat{R}$ denote the Riemann curvature tensor of $\Sigma$, and denote by $K$ and $\tilde{K}$, respectively, the second fundamental form of $\Sigma$ as a hypersurface of $\mathbb{R}^n$ and the second fundamental form of $f(\Sigma)$ as a hypersurface of $(M,g)$. By equations (4.3) and (4.4) of \cite{LWX}, we have
$$
 \tilde{K} = \frac{K}{\sqrt{1 + |Df|^2}}.
$$
and
$$
R_{AB}^{\ \ \ \ CD} = \frac{|Df|^2}{1 + |Df|^2}\hat{R}_{AB}^{\ \ \ \ CD}.
$$
Thus, plugging into the expression for $\left( T_{(2q-1)\alpha} \right)^1_1$, we find
\begin{align*}
\left( T_{(2q-1)\alpha} \right)^1_1 
  = \ & \frac{1}{2^{q-1}}\frac{1}{(2q-1)!} \left( \frac{|Df|^2}{1 + |Df|^2} \right)^{q-1} \\
      & \times \delta^{A_1 \cdots A_{2q-1}}_{B_1 \cdots B_{2q-1}}
        \left( \prod_{s=1}^{q-1} \hat{R}\indices{_{A_{2s-1}A_{2s}}^{B_{2s-1}B_{2s}}} \right)
        \left(A_{\alpha}\right)_{A_{2q-1}}^{B_{2q-1}}.
\end{align*}
From
$$\eta^{\alpha} = e_{\alpha} - Df^{\alpha} = e_{\alpha} - \langle Df^{\alpha}, \xi \rangle \xi,$$
it follows that
$$ \left( A_{\alpha} \right)_A^B = - \langle Df^{\alpha}, \xi \rangle K_A^B.$$
Also, the Gauss equation applied to $\Sigma \subset \mathbb{R}^n$ yields
$$ \hat{R}_{AB}^{\ \ \ \ CD} = K_A^C K_B^D - K_A^D K_B^C.$$
We then conclude that
\begin{align*}
\left( T_{(2q-1)\alpha} \right)^1_1 
  = & -\frac{1}{2^{q-1}}\frac{1}{(2q-1)!} \langle Df^{\alpha}, \xi \rangle
      \left( \frac{|Df|^2}{1 + |Df|^2} \right)^{q-1} \\
    & \times \delta^{A_1 \cdots A_{2q-1}}_{B_1 \cdots B_{2q-1}}
    \left( \prod_{s=1}^{q-1} \hat{R}\indices{_{A_{2s-1}A_{2s}}^{B_{2s-1}B_{2s}}} \right)
    \left(K\right)_{A_{2q-1}}^{B_{2q-1}} \\
  = & -\frac{1}{(2q-1)!} \langle Df^{\alpha}, \xi \rangle
      \left( \frac{|Df|^2}{1 + |Df|^2} \right)^{q-1} \\
    & \times \delta^{A_1 \cdots A_{2q-1}}_{B_1 \cdots B_{2q-1}}
      \left( \prod_{s=1}^{q-1} K_{A_{2s-1}}^{B_{2s-1}}K_{A_{2s}}^{B_{2s}} \right) K_{A_{2q-1}}^{B_{2q-1}} \\
  = & \  -\langle Df^{\alpha}, \xi \rangle \left( \frac{|Df|^2}{1 + |Df|^2} \right)^{q-1} H_{(2q-1)},
\end{align*}
where we have used the expression \eqref{eq:r-mean-curvature} to obtain the last equality.
It follows that
\begin{align*}
\langle X_{(q)}, \xi \rangle
  & = \frac{1}{2}(2q-1)! \frac{\langle Df^{\alpha}, \xi \rangle}{1 + |Df|^2} \left( T_{(2q-1)\alpha} \right)^1_1  \\
  & = - \frac{1}{2}(2q-1)!\frac{\langle Df^{\alpha}, \xi \rangle^2}{1 + |Df|^2}
    \left( \frac{|Df|^2}{1 + |Df|^2} \right)^{q-1}H_{(2q-1)} \\
  & = - \frac{1}{2}(2q-1)!\left( \frac{|Df|^2}{1 + |Df|^2} \right)^{q}H_{(2q-1)},
\end{align*}
where, to obtain the last equality, we have used that
$$ Df^{\alpha} = \langle Df^{\alpha}, \xi \rangle \xi$$
implies
$$ |Df|^2 = \sum_{\alpha} \langle Df^{\alpha}, \xi \rangle^2.$$
\end{proof}

\begin{rmk}
{\rm In Proposition \ref{prop 3}, the expression $|Df|^2 / \left( 1 + |Df|^2 \right)$
is the cosine of the angle between the graph and the hyperplane containing its boundary (see \cite{deS}).
}

\end{rmk}

\section{Proof of the theorems }

Suppose first that $M$ has no boundary. Let $S_r$ be an Euclidean coordinate sphere of radius $r$. By (\ref{GBC mass}), (\ref{vector field}) and the divergence theorem we have
\begin{align*}
\mathrm{m}_q 
= & \ c_q (n) \lim_{r \to \infty} \int_{S_r} X_{(q)}^i\xi_i dS_r \\
= & \ c_q(n)  \int _{\mathbb{R}^n} \mathrm{div}_e X_{(q)} dV,
\end{align*}
where $dV$ denotes the Euclidean volume form. Thus, invoking Proposition \ref{field_div} and using that
\begin{equation} \label{dV}
dV = \frac{1}{\sqrt{G}} dM,
\end{equation}
we find
$$ \mathrm{m}_q = \frac{1}{2} c_q(n) \int_M \left(L_{(q)} + (2q-1)! \langle \left[ T_{(2q-1)\alpha}, A_{\beta} \right] \cdot e_{\alpha}^{\top}, e_{\beta}^{\top} \rangle \right) \frac{1}{\sqrt{G}}  dM,$$
which is exactly the first part of Theorem \ref{main 1}.

To prove the second part of Theorem \ref{main 1}, notice that, from equations 3 and 6 of \cite{AKN}, the tensor $T_{(2q-1)\alpha}$ can be written as a polynomial on the $A_{\alpha}$'s.
Also, if $M$ has flat normal bundle, then the Ricci equation (\ref{Ricci}) yields
\begin{equation} \label{comutation}
\left[ A_{\alpha}, A_{\beta} \right] = 0,
\end{equation} 
for all $\alpha, \beta \in \lbrace 1, \ldots, m \rbrace$.
Thus, using (\ref{comutation}) several times, we find
$$\left[ T_{(2q-1)\alpha}, A_{\beta} \right] = 0,$$
for all $\alpha, \beta \in \lbrace 1, \ldots, m \rbrace$. Hence, equation (\ref{eq main 1}) becomes
$$ \mathrm{m}_q = \frac{1}{2} c_q(n) \int_M L_{(q)}  \frac{1}{\sqrt{G}}  dM.$$
Therefore, if $L_{(q)}$ is nonnegative, then $\mathrm{m}_q$ is nonnegative. This finishes the proof of Theorem \ref{main 1}.

Suppose now that $\partial M$ is not empty and that $f$ can be extended to a smooth map on some open set containing $\mathbb{R}^n \setminus \Omega$. This assumption allows us to use the results of Section \ref{sec:1}. Equations (\ref{GBC mass}), (\ref{vector field}) and the divergence theorem yield
\begin{align*}
\mathrm{m}_q 
= & \ c_q (n) \lim_{r \to \infty} \int_{S_r} X_{(q)}^i\nu_i dS_r   \\
= & \ c_q(n)  \int _{\mathbb{R}^n \setminus \Omega} \mathrm{div}_e X_{(q)} dV - c_q(n)\int_{\Sigma} \langle X_{(q)}, \xi \rangle d\Sigma.
\end{align*}
Invoking Proposition \ref{field_div}, Proposition \ref{prop 3} and equation (\ref{dV}), we get
\begin{align}
\mathrm{m}_q 
= & \ \frac{1}{2} c_q(n) \int_M \left(L_{(q)} + (2q-1)! \langle \left[ T_{(2q-1)\alpha}, A_{\beta} \right] \cdot e_{\alpha}^{\top}, e_{\beta}^{\top} \rangle \right) \frac{1}{\sqrt{G}}  dM \nonumber \\
{ } & \ + \frac{1}{2} (2q-1)! c_q(n) \int_{\Sigma} \left( \frac{|Df|^2}{1 + |Df|^2} \right)^q H_{(2q-1)} d\Sigma. \label{quase}
\end{align}
This finishes the proof of Theorem \ref{main 1.5}.

Let us now prove Theorem \ref{main 2}. We cannot use equation (\ref{quase}) directly, since, by hypothesis, 
$$ |Df| \to \infty \ \ \textrm{as} \ \ x \to \Sigma,$$
and hence, it is not possible to extend $f$ to a smooth function on some open set containing $\mathbb{R}^n \setminus \Omega$. To circumvent this problem, we proceed as in the last section of \cite{MV}. Namely, we consider an approximating sequence 
$$F^k = \left(f^{1,k}, \ldots, f^{m,k}\right) : \mathbb{R}^n \setminus \Omega \to \mathbb{R}^m,$$ $k \in \mathbb{N}$, of smooth maps such that each $F^k$ extends to a smooth map on some open set containing $\mathbb{R}^n \setminus \Omega$. We then apply (\ref{quase}) to each $F^k$ and take the limit as $k \to \infty$, reaching (\ref{eq main 2}).

It remains to prove inequality (\ref{eq_2 main 2}). If $\Sigma$ has only one component then, by a result of Guan and Li \cite{GL}, it holds
\begin{equation} \label{guan-li}
 \frac{1}{2} (2q-1)! c_q(n) \int_{\Sigma} H_{(2q-1)} d\Sigma \geq \frac{1}{2^q} 
\left( \frac{|\Sigma|}{\omega_{n-1}} \right)^{\frac{n-2q}{n-1}},
\end{equation}
with equality holding if and only if $\Sigma$ is a round sphere.

Suppose now that $\Sigma$ has more than one component. Recall that if $x_1, \ldots, x_m$ are nonnegative real numbers and $0 \leq s < 1$ then
\begin{equation} \label{s}
\sum_{i = 1}^m x_i^s \geq \left( \sum_{i=1}^m x_i \right)^s,
\end{equation}
with equality holding if and only if at most one of the $x_i$'s is positive (see \cite{HW}, Proposition 5.2). Inequality (\ref{eq_2 main 2}) then follows by combining inequalities (\ref{guan-li}) and  (\ref{s}).

\begin{rmk} {\rm Unfortunately our methods are not suitable to deal with the equality cases, that is, to prove the rigidity statements contained in conjectures \ref{PMT for m_q} and \ref{PI for m_q}. If equality holds in Theorem \ref{main 1}, then we can only conclude that the Gauss--Bonnet curvature $L_{(q)}$ is identically zero. If equality holds in Theorem \ref{main 2}, then we can only conclude that $L_{(q)}$ is zero on $M$ and that $\Sigma$ has only one component which is a round sphere.}
\end{rmk}

\begin{multicols}{2}
\noindent
Alexandre de Sousa \\
Universidade Federal de Alagoas \\
{\tt a.asm@protonmail.com} \\
Frederico Gir\~ao \\
Universidade Federal do Cear\'a \\
{\tt fred@mat.ufc.br} 
\end{multicols}

\end{document}